\documentclass{amsart}
\usepackage{amsmath,amssymb,amsthm,amsfonts,enumerate,array,color,lscape,fancyhdr,layout,pst-all}
\usepackage[english]{babel}
\usepackage[latin1]{inputenc}
\usepackage{young}

\input xy
\xyoption{all}

\usepackage[hcentermath]{youngtab}
\usepackage{graphicx}
\usepackage{float}
\usepackage{pstricks}

\newcounter{numberofremark}
\newcommand\nothing[1]{}
\usepackage[active]{srcltx}
\usepackage[hyperindex,bookmarksnumbered,plainpages]{hyperref}

\newcommand{\dcl}{\DeclareMathOperator}
\dcl\cdet{cdet} \dcl\Sp{Specm} \dcl\depth{depth} \dcl\im{Im} \dcl\Span{span} \dcl\Ker{Ker} \dcl\Specm{Specm}
\dcl\Supp{Supp} \dcl\codim{codim} \dcl\Y{Y} \dcl\gl{\mathfrak{gl}}    \dcl\U{U} \dcl\T{T}
\dcl\qdet{qdet} \dcl\sgn{sgn} \dcl\gr{gr} \dcl\diag{diag}
\dcl\g{\mathfrak{g}} \dcl\C{\mathbb C} \dcl\dd{{\mathrm d}}

\dcl\End{End} 

\newcommand\sn{{\mathsf n}}
\newcommand\sm{{\mathsf m}}

\newcommand\Ga{{\Gamma}}

\newlength\yStones
\newlength\xStones
\newlength\xxStones

\makeatletter
\def\Stones{\pst@object{Stones}}
\def\Stones@i#1{%
  \pst@killglue%
  \begingroup%
  \use@par%
  \setlength\xxStones{\xStones}%
  \expandafter\Stones@ii#1,,\@nil
  \endgroup
  \global\addtolength\xStones{0.6cm}%
  \global\addtolength\yStones{-7.5mm}}%
\def\Stones@ii#1,#2,#3\@nil{%
  \rput(\xxStones,\yStones){%
    \psframebox[framesep=0]{%
      \parbox[c][6mm][c]{11mm}{\makebox[11mm]{$#1$}}}}%
  \addtolength\xxStones{1.2cm}%
  \ifx\relax#2\relax\else\Stones@ii#2,#3\@nil\fi}
\makeatother

\def\Stone#1{\fbox{\makebox[13mm]{\strut#1}}\kern2pt}

\newtheorem{theorem}{Theorem}[section]
\newtheorem{lemma}[theorem]{Lemma}
\newtheorem{corollary}[theorem]{Corollary}
\newtheorem{proposition}[theorem]{Proposition}
\newtheorem{example}[theorem]{Example}
\newtheorem{remark}[theorem]{Remark}

\begin{document}
\title[Drinfeld category and singular  Gelfand-Tsetlin modules]{Drinfeld category and the classification of singular  Gelfand-Tsetlin $\displaystyle \mathfrak{gl}_n$-modules}

\author[V.Futorny]{Vyacheslav Futorny}

\address{Instituto de Matem\'atica e Estat\'istica, Universidade de S\~ao
Paulo,  S\~ao Paulo SP, Brasil} \email{futorny@ime.usp.br,}
\author[D.Gratcharov]{Dimitar Grantcharov}
\address{\noindent
University of Texas at Arlington,  Arlington, TX 76019, USA} \email{grandim@uta.edu}
\author[L.E.Ramirez]{Luis Enrique Ramirez}
\address{Universidade Federal do ABC,  Santo Andr\'e-SP, Brasil} \email{luis.enrique@ufabc.edu.br,}

\begin{abstract}
We prove a uniqueness theorem for  irreducible non-critical Gelfand-Tsetlin modules. The uniqueness result leads to  a
complete classification of the irreducible Gelfand-Tsetlin modules with $1$-singularity. An explicit construction of  such modules was given  in \cite{FGR2}. In particular, we show  that the modules constructed in \cite{FGR2}
 exhaust all irreducible  Gelfand-Tsetlin modules with   $1$-singularity. To prove the result we introduce a new category of modules (called Drinfeld category) related to the Drinfeld generators of the Yangian $Y(\gl_n)$ and  define a functor from the category  of non-critical Gelfand-Tsetlin modules to the Drinfeld category. 
\end{abstract}

\maketitle

\section{Introduction} \label{sec-intro}

The representation theory of the Lie algebra $\g = \gl_n$ of all $n\times n$ complex matrices plays fundamental role in numerous areas of mathematics and physics.  A  natural category of  $\g$-modules is the category  of  Gelfand-Tsetlin modules - those  that have generalized eigenspace decompositions over a certain maximal commutative subalgebra (\emph{Gelfand-Tsetlin subalgebra}) $\Gamma$ of the universal enveloping algebra of $\g$. The concept of a Gelfand-Tsetlin module generalizes  the construction of  classical Gelfand-Tsetlin bases that have been   introduced in \cite{GG}, \cite{GT}. The general theory of Gelfand-Tsetlin modules was developed in  \cite{CE1}, \cite{CE2},  \cite{DFO2}, \cite{DFO3},  \cite{FGR1}, \cite{FGR2}, \cite{FGR4},   \cite{Gr1}, \cite{Gr2},   \cite{GS},  \cite{KW-1}, \cite{KW-2}, \cite{Maz1}, \cite{Maz2},   \cite{m:gtsb}, \cite{Ovs}, among  others.

Throughout the paper, for  $n\geq 2$, $T_n(\mathbb C) \simeq {\mathbb C}^{\frac{n(n+1)}{2}}$ will stand for the space consisting of the following \emph{Gelfand-Tsetlin tableau}: \\
  
  \
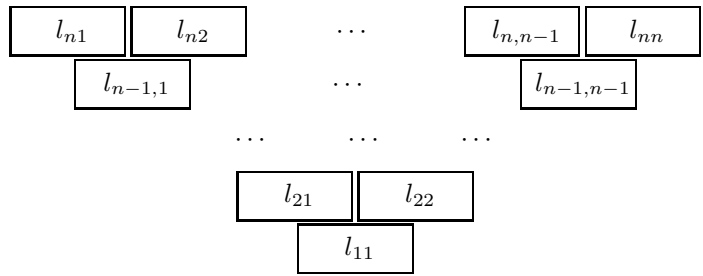
\begin{center}
\Stone{\mbox{ $l_{n1}$}}\Stone{\mbox{ $l_{n2}$}}\hspace{1cm} $\cdots$ \hspace{1cm} \Stone{\mbox{ $l_{n,n-1}$}}\Stone{\mbox{ $l_{nn}$}}\\[0.2pt]
\Stone{\mbox{ $l_{n-1,1}$}}\hspace{1.7cm} $\cdots$ \hspace{1.8cm} \Stone{\mbox{ $l_{n-1,n-1}$}}\\[0.3cm]
\hspace{0.2cm}$\cdots$ \hspace{0.8cm} $\cdots$ \hspace{0.8cm} $\cdots$\\[0.3cm]
\Stone{\mbox{ $l_{21}$}}\Stone{\mbox{ $l_{22}$}}\\[0.2pt]
\Stone{\mbox{ $l_{11}$}}\\
\medskip
\end{center}

 We will identify  $ T_{n}(\mathbb{C})$ with the set  ${\mathbb C}^{\frac{n(n+1)}{2}}$ in the following way:
 to  $$
L=(l_{n1},...,l_{nn}|l_{n-1,1},...,l_{n-1,n-1}| \cdots|l_{21}, l_{22}|l_{11})\in {\mathbb C}^{\frac{n(n+1)}{2}}
$$ 
 we associate a tableau   $T(L)$ as above.

For a fixed element 
$L=(l_{ij})_{j\leq i=1}^n$  in $T_n(\mathbb C)$ consider the set 
$$L_{\mathbb Z}:=L+ T_{n-1}(\mathbb Z)=\{L+M\; | \; M=(m_{ij})_{j\leq i=1}^n\in T_n(\mathbb Z), m_{nk}=0 \,, k=1, \ldots, n \}.$$
Henceforth, we define $V(L)$ to be the  complex vector space with basis the set $L_{\mathbb Z}$, i.e. $V(L) = \bigoplus_{L' \in L_{\mathbb Z}} {\mathbb C} T(L')$. Note that if $L', L'' \in L_{\mathbb Z}$, then $T(L'+L'') \neq T(L') + T(L'')$ in $V(L)$ (even if $L' + L'' \in L_{\mathbb Z}$).

An open problem which dates back to the work of Gelfand and Graev, \cite{GG}, is  to define a $\mathfrak{gl}_n$-module structure on a subspace of  $V(L)$  annihilated by some maximal  ideal  of the Gelfand-Tsetlin subalgebra.
Solutions to this problem are known 
in various cases, the most fundamental of which relies on  the  construction of Gelfand-Tsetlin bases of finite dimensional representations of $\g$.

Another important solution concerns the generic Gelfand-Tsetlin modules. More precisely, a  pair of entries $(l_{mi},l_{mj})$ of $L$ such that  $l_{mi}-l_{mj}\in \mathbb Z$ is called a {\it singular pair}. If $T(L)$ contains a singular pair in the $m$-th row for some $2\leq m\leq n-1$
then we call $T(L)$ a \emph{singular tableau} (and $L$ a \emph{singular} element in $T_n({\mathbb C})$). A tableau $T(L)$ is {\it $1$-singular} if it contains exactly one singular pair. If $T(L)$ is not singular it is called \emph{generic}. For a generic tableau $T(L)$, one can imitate the construction of Gelfand-Tsetlin bases of finite-dimensional representations and construct a $\mathfrak{gl}_n$-module structure on $V(L)$, see for example  \cite{DFO3}.

The  study of singular  modules $V(L)$ was initiated in \cite{FGR2}, where a module structure on $V(L)$ 
was introduced for any $1$-singular tableau $L$. The latter construction was generalized in \cite{FGR4} for singular tableaux $L$ with multiple singular pairs such that the difference of the entries of any two  distinct singular pairs is noniteger.  By taking irreducible quotients of $V(L)$ new irreducible  Gelfand-Tsetlin modules for $\gl_n$ were constructed. In particular,  understanding  the $1$-singular case helped us to  complete  classification of irreducible  Gelfand-Tsetlin $\mathfrak{gl}_3$-modules in \cite{FGR3}. Explicit basis for a class of irreducible  $1$-singular Gelfand-Tsetlin $\mathfrak{gl}_n$-modules was constructed in \cite{GR}.

We note that to each  $L \in T_n (\mathbb C)$ we associate the maximal ideal $\sm_L$ of the Gelfand-Tsetlin subalgebra $\Gamma$ generated by $c_{ij} - \gamma_{ij}(L)$, $1 \leq j \leq i \leq n$, where $c_{ij}$ are the generators of $\Gamma$, and $\gamma_{ij}$ are  symmetric polynomials defined in (\ref{def-gamma}).    
Furthermore, for every maximal ideal $\sm$ of $\Ga$  there exists an irreducible Gelfand-Tsetlin module $V$ such that $V_{\sm} \neq 0$ (see (\ref{gt-def}) for the definition of $V_{\sm}$) and the number of such non-isomorphic irreducible modules $V$ is finite, \cite{Ovs}.   In fact, for a generic $L$  there exists a unique up to isomorphism irreducible module $V$  such that  $V_{\sm_{L}}\neq 0$. On the other hand, if $L$ is $1$-singular, the number of such 
non-isomorphic irreducible Gelfand-Tsetlin modules is bounded by $2$ and it is $2$ for some $L$.   The following conjecture was stated in \cite{FGR2} and was known to be true for  $n=2$ and $n=3$.

\

\noindent {\bf Conjecture.}
 Let $L$ be a $1$-singular tableau and $\sm=\sm_{L}$. Then 
  any irreducible Gelfand-Tsetlin module $V$ with $V_{\sm_{L}}\neq 0$ is isomorphic to a subquotient of $V(L)$. 
  
  \

 The main purpose of the current paper is to prove this conjecture and, hence, to complete the classification of all irreducible $1$-singular Gelfand-Tsetlin modules of $\gl_n$,   $n \geq 2$.  The nontrivial case of the conjecture concerns non-critical modules, see Theorem \ref{theorem-fgr2}. We say that $L$ and $\sm=\sm_L$ are \emph{critical} if $L$ has equal entries in some of its rows.  A Gelfand-Tsetlin module $V$ such that $V_{\sm} \neq 0$ for some critical maximal ideal $\sm$ will be called  \emph{critical}.   
 
 To prove the conjecture we introduce a new technique, which is essentially different from the one we used in \cite{FGR2} and \cite{FGR4}. The main idea is to consider Drinfeld generators of the universal enveloping algebra of $\gl_n$. With the aid of the Drinfeld generators, for every non-critical maximal ideal $\sm$ of $\Gamma$  we introduce a new category of modules $\mathcal D_{\sm}$-mod that we call \emph{Drinfeld category}. Then we define a faithful functor from the category of Gelfand-Tsetlin modules having $\sm$ as a character  to  $\mathcal D_{\sm}$-mod. This functor
  preserves irreducibility and maps non-isomorphic modules to non-isomorphic.

\medskip

  \noindent {\bf Main Theorem.}  {\emph{Let $V$ be an irreducible non-critical Gelfand-Tsetlin $\gl_n$-module. Then the following hold.}
 \begin{itemize}
\item[(i)] {\it The module $V$ is tame, i.e. $\Gamma$ acts diagonally on $V$.}
\item[(ii)] {\it Let $\sm$ be a maximal ideal such that $V_{\sm} \neq 0$. If $W$ is an irreducible non-critical Gelfand-Tsetlin $\gl_n$-module such that $W_{\sm} \neq 0$, then $W$ is isomorphic to $V$.}  
\end{itemize}

\

The proof of Main Theorem involves a study of the irreducible modules in $\mathcal D_{\sm}$-mod. Main Theorem, in particular,  implies the conjecture above and completes the classification of the irreducible $1$-singular $\g$-modules. 

The organization of the paper is as follows. In Section 2 we collect some important results about singular Gelfand-Tsetlin modules. The definition of the Drinfeld category $\mathcal D_{\sm}$-mod and the functor from the category of Gelfand-Tsetlin modules having $\sm$ in its support to $\mathcal D_{\sm}$-mod  are included in Section 3. In Section 4 we study irreducible modules in  $\mathcal D_{\sm}$-mod. The main result is proven in Section 5.

\

\noindent{\bf Acknowledgements.}  V.F. is
supported in part by  CNPq grant (301320/2013-6) and by 
Fapesp grant (2014/09310-5). D.G is supported in part by Simons Collaboration Grant 358245.  Part of this work was done during the visit of the first author to the University of Sydney whose hospitality and support are greatly appreciated.
The authors would like to thank Alexander Molev for the numeropus helpful discussions and, in particular, for suggesting to use Drinfeld generators. The authors are also grateful to the referee for improving the exposition of the paper.

\section{Singular Gelfand-Tsetlin modules} \label{sec-singular-gt}
Consider  a chain of embeddings of Lie subalgebras 
$$\mathfrak{gl}_{1}\subset \mathfrak{gl}_{2}\subset \ldots \subset \mathfrak{gl}_n.$$
The choice of embeddings is not essential but for simplicity we chose embeddings of principal submatrices of $n\times n$-matrices.
Set $U_k=U(\mathfrak{gl}_k)$ to be the universal enveloping algebra of $\mathfrak{gl}_k$,  $k=1, \ldots, n$. 
 The {\it Gelfand-Tsetlin
subalgebra} $\Ga$ of $U=U_n$ associated with this chain of embeddings is  the subalgebra generated by the centers  of universal enveloping algebras of  $\gl_k$, $k=1, \ldots, n$. It is the  polynomial algebra in  the $\displaystyle \frac{n(n+1)}{2}$ generators
 $\{ c_{mk}\,|\,1\leq k\leq m\leq n \}$, where
\begin{equation}\label{equ_3}
c_{mk } \ = \ \displaystyle {\sum_{(i_1,\ldots,i_k)\in \{
1,\ldots,m \}^k}} E_{i_1 i_2}E_{i_2 i_3}\ldots E_{i_k i_1}, 
\end{equation}
 and $E_{ij}$ denote the $(i,j)$th elementary matrix, \cite{Zh}.

A $\mathfrak{gl}_n$-module $V$ is a \emph{Gelfand-Tsetlin  module} (with respect to $\Ga$) if
 \begin{equation} \label{gt-def}
V=\bigoplus_{\sm\in\Sp\Ga}V_{\sm},
\end{equation}
where $$V_{\sm}=\{v\in V\; | \; \sm^{k}v=0 \text{ for some }k\geq 0\},$$
and $\Sp\Ga$ denotes the set of maximal ideals of $\Ga$. The set of all $\sm$ in $\Sp\Ga$ for which $V_{\sm} \neq 0$ is called the \emph{Gelfand-Tsetlin support} or just  the \emph{support} of $V$. 

 To every element $L = (l_{ij})$ in $T_n ({\mathbb C})$  we  associate the maximal ideal $\sm_{L}$ of $\Ga$ generated by $c_{mk}-\gamma_{mk}(L)$ where 
\begin{equation} \label{def-gamma}
\gamma_{mk} (L): = \ \sum_{i=1}^m
(l_{mi}+m-1)^k \prod_{j\ne i} \left( 1 -
\frac{1}{l_{mi}- l_{mj}} \right).
\end{equation} 

Recall the definition of $L_{\mathbb Z}$ from the introduction. For a fixed maximal ideal $\sm_L$ of $\Ga$ the number of $L'\in L_{\mathbb Z}$  such that  $\sm_{L'} = \sm_L$  is finite, and all such $L'$ belong to a single orbit of the group
$G=S_1\times \ldots \times S_n$ (each $S_m$ acts on the $m$-th row of $L'$ by permutations).   

Recall the definition of a singular tableau from the introduction. We will say that $T(L)$ is a \emph{critical tableau} if the two entries of the singular pair of $L$ are equal. We call a maximal ideal $\sm = \sm_L$ of $\Ga$ \emph{generic}, \emph{singular}, or \emph{critical}, if $T(L)$ is generic, singular, or critical, respectively. Recall that the classification of simple Gelfand-Tsetlin modules having only generic tableaux in their support has been completed in \cite{DFO3}. The following theorem shows that the classification of simple $1$-singular Gelfand-Tsetlin modules $V$ can be reduced to those $V$ that have no critical ideals in their Gelfand-Tsetlin support. 

\begin{theorem} \label{theorem-fgr2}\cite{FGR2}  Let $L \in T_{n} (\mathbb C)$ be $1$-singular.  Then the following hold. \begin{itemize}
 \item[(i)]   $V(L)$ has a structure of a Gelfand-Tsetlin  $\mathfrak{gl}_n$-module.
  \item[(ii)] For any $L'\in L_{\mathbb Z}$ there exists an irreducible subquotient $V$ of $V(L)$ such that $V_{\sm_{L'}}\neq 0$.
  The number of all such irreducible subquotients  $V$ is bounded by $2$.
  \item[(iii)] If $L$ is a critical tableau then there exists a unique (up to an isomorphism)
 irreducible Gelfand-Tsetlin module   $V$ such that $V_{\sm_{L}}\neq 0$. 
  \item[(iv)] $V(L)\simeq V(L')$ if and only if there exists $\sigma\in G$ such that $L-\sigma(L')\in T_{n-1}(\mathbb Z)$, or equivalently, if $L$ and $L'$ are in the same orbit under the action of $G \ltimes T_{n-1} ({\mathbb Z})$ on $T_{n} (\mathbb C)$.
  \item[(v)] $\dim V(L)_{\sm}\leq 2$ for any maximal ideal $\sm$ of $\Ga$. If 
\end{itemize}
\end{theorem}

\section{Drinfeld category}

\subsection{Drinfeld generators}
We briefly recall the definition of the Yangian of $\gl_n$. For more details we refer the reader to the original work of Drinfeld, \cite{D}, as well as to the  book \cite{m:gtsb}. The Yangian $Y(\gl_n)$ 
of  $\gl_n$ is an associative algebra generated by
the elements $T_{ij}^{(r)}$ where $i,j = I, \ldots, n$ and $r = 1, 2, \ldots$ subject to the following
relations:
$$[ T^{(r)}_{ij} , T^{(s+1)}_{kl} ] - [ T^{(r+1)}_{ij} , T^{(s)}_{kl} ] = T^{(r)}_{kj }T^{(s)}_{il} - T^{(s)}_{kj} T^{(r)}_{il},\, T^{(0)}_{ij} = \delta_{ij}$$ 
$ r, s = 0, 1, 2, \ldots$. Note that we fix the relations to be the same as those in \cite{NT}. Introduce the formal Laurent series in $u^{-1}$,
$$T_{ij}(u) = T^{(0)}_{ij}u + T^{(1)}_{ij} + T^{(2)}_{ij} u^{-1} + T^{(3)}_{ij} u^{-2} + \ldots$$
and set $$
T(u) :=
\sum^n_{i,j=1}
E_{ij} \otimes T_{ij}(u).$$

Next we recall the definition of  quantum determinant.
Let $X(u) = (X_{ij}(u))_{i,j=1}^m$ be an
arbitrary matrix whose entries are formal Laurent series in $u^{-1}$ with coefficients
in $Y(\gl_n)$. The \emph{quantum determinant} of $X(u)$ is:
$$\qdet X(u) := \sum_{\sigma\in S_m} (-1)^{l(\sigma)}X_{1\sigma(1)}(u)X_{2\sigma(2})(u - 1)\ldots X_{m,\sigma(m)}(u - m + 1),$$
where  $l(\sigma)$ is the length of $\sigma$.

For each $m=1, \ldots, n$ set $A_m(u) = \qdet [T_{ij}(u)]^m_{i,j=1}$.
The coefficients of $u^{0}$, $u^{-1}$, $\ldots$ in the expansion of $A_m(u)$ are free
generators of the center of the algebra $Y(\gl_m)$ for $m=1, \ldots, n$, \cite{MNO}. Moreover, all coefficients of $A_1(u), \ldots, A_n(u)$ pairwise commute. 

For any $m = 1, \ldots, n - 1$ denote by $B_m(u)$, $C_m(u)$,  $D_m(u)$
 the quantum determinants of the submatrices of $T(u)$ generated respectively: by the rows $1, \ldots, m - 1,
m+1$ and the columns $1, \ldots, m$; by
 the rows $1, \ldots, m$ and the columns $1, \ldots, m - 1,  m + 1$;  by the rows $1, \ldots, , m- 1, m+1$ and the 
columns $1, \ldots, , m- 1, m+1$. The coefficients of $A_m(u), B_m(u)$, $C_m(u)$,  and $D_m(u)$
form an alternative set of generators for $Y(\gl_n)$, \cite{D}, called \emph{Drinfeld generators}. 
Note that our $B_m(u)$ and $C_m(u)$ correspond to $C_m(u)$ and $B_m(u)$, respectively  in \cite{NT}. 

In the following proposition we collect some basic properties of the Drinfeld generators.

\begin{proposition}[\cite{NT}, Proposition 1.2, Proposition 1.4] \label{prop-formulas}
The following commutation relations hold in $Y(\gl_n)$:
\begin{itemize}
\item[(i)] $[A_m(u), B_l(v)] = 0$ if $l\neq m$,
\item[(ii)] $[A_m(u), C_l(v)] = 0$ if $l\neq m$,
\item[(iii)] $[C_m(u), B_l(v)] = 0$ if $l\neq m$,
\item[(iv)] $[S_m(u), S_l(v)] = 0$ if $|l -m| \neq 1$, $S\in \{B, C\}$,
\item[(v)]  $(u - v)  [A_m(u), C_m(v)] = C_m(u)A_m(v) - C_m(v)A_m(u)$,
\item[(vi)] $(u - v)  [A_m(u), B_m(v)] =- B_m(u)A_m(v) + B_m(v)A_m(u)$,
\item[(vii)] $(u - v) [C_m(u), B_m(v)] = - D_m(u)A_m(v) + D_m(v)A_m(u)$,
\item[(viii)] $B_m(u)C_m(u - 1) = D_m(u)A_m(u - 1) - A_{m+1}(u)A_{m-1}(u - 1).$
\item[(ix)] $C_m(u-1)B_m(u) = D_m(u-1)A_m(u) - A_{m+1}(u)A_{m-1}(u - 1).$
\end{itemize}
\end{proposition}

Now, following the construction of \cite{NT}, we define generators of $U(\gl_n)$ coming from  $A_m(u)$, $B_m(u)$, $C_m(u)$, $D_m(u)$, as follows. The algebra $Y(\gl_n)$ contains the universal enveloping $U(\gl_n)$ as a subalgebra  through the embedding $E_{ij} \mapsto T_{ji}^{(1)}$. Also, there exists a surjective homomorphism $\phi: Y(\gl_n)\rightarrow U(\gl_n)$ defined  by
$$T_{ij}(u)\mapsto \delta_{ij}u+ E_{ji}.$$ 
We denote the images of $A_m(u)$, $m=1, \ldots, n$ and $B_m(u)$, $C_m(u)$, $D_m(u)$, $m=1, \ldots, n-1$, under $\phi$, by 
 $a_m(u)$,  $b_m(u)$, $c_m(u)$ and $d_m(u)$, respectively. These images are polynomials of degree $m$, $m-1$, $m-1$, $m$, respectively, and their coefficients generate $U(\gl_n)$. We call these coefficients the \emph{Drinfeld generators of} $U(\gl_n)$. For instance, $E_{m,m+1}$ is the leading coefficient of $b_m(u)$, while $E_{m+1,m}$ is the leading coefficient of $c_m(u)$.
The coefficients of  $a_m (u)$, $m=1, \ldots, n$, generate the Gelfand-Tsetlin subalgebra $\Gamma$.  Clearly, the Drinfeld generators of $U(\gl_n)$ satisfy the relations listed in Proposition \ref{prop-formulas}. 



\subsection{Interpolation} \label{subsec-interpol}

For any $m=1, \ldots, n-1$ and any tuple $(l_{m1}, \ldots, l_{mm})$ of distinct complex numbers we can recover the Drinfeld generators  of  $U(\gl_n)$ knowing their values at  $-l_{mi}$, $i=1, \ldots, m$ using the Lagrange interpolation formula.  For instance, 
$$b(u)=\sum_{i=1}^{n-1}\frac{(u+l_{m1})\ldots \widehat{(u+l_{mi})} \ldots (u+l_{mm})}{(l_{mi}-l_{m1})\ldots \widehat{(l_{mi}-l_{mi})} \ldots (l_{mi}-l_{mm})}b(-l_{mi})$$ 
(as usual, $\widehat{t}$ indicates that the corresponding term $t$ is missing in the product). For a general tuple $(l_{m1}, \ldots, l_{mm})$  one can use the Lagrange-Newton interpolation formula. We note though that in this paper we  restrict our attention to non-critical tableaux $L$ (distinct $l_{m1},...,l_{mm}$).

Let now $V$ be a Gelfand-Tsetlin module such that $V$ has no critical  maximal ideals in its Gelfand-Tsetlin support. In what follows we describe how we can define the action of the Drinfeld generators on $V$ based on the interpolation formulas written above.  Let $\sm$ be in the Gelfand-Tsetlin support of $V$, i.e. $\sm$ is a maximal ideal in  $\Ga$ such that $V_{\sm}\neq 0$. Choose an arbitrary tableau $T(L)$ such that $\sm=\sm_L$ and whose $m$th row is $(l_{m1}, \ldots, l_{mm})$. The action of $a_m(u)$, $m=1, \ldots, n$, on $V_{\sm}$ is determined by $\sm$ only. More precisely, there is a polynomial $\alpha_m(u, \sm)$ of $u$ depending on $\sm$ such that  $a_m(u) - \alpha_m(u, \sm)\mbox{Id}$ acts finitely on $V_{\sm}$.
 To define the action of the Drinfeld generators $b_m(u)$ and $c_m(u)$ on $V_{\sm}$  it is sufficient to define 
the action of the operators $b_m(-l_{mi})$ and  $c_m(-l_{mi})$, $i=1, \ldots, m$. For the latter operators we have the following important properties. In what follows, by $L-\delta_{mi}$ (respectively, $L +\delta_{mi}$) we denote the element of $L_{\mathbb Z}$ obtained from $L$ after substracting $1$ from (respectively, adding $1$ to) $l_{mi}$.

\begin{proposition} \label{prop-der-crit} For any  $L = (l_{ij})$ and $\sm = \sm_L$ the following hold. 
\begin{itemize}
\item[(i)] $c_m(-l_{mi})V_{\sm}\subset V_{\sm'}$, where $\sm'=\sm_{L-\delta_{mi}}$;
\item[(ii)] $b_m(-l_{mi})V_{\sm}\subset V_{\sm''}$, where $\sm''=\sm_{L+\delta_{mi}}$.
\end{itemize}
\end{proposition}

On the other hand, Proposition \ref{prop-formulas} implies  other useful properties for the action of $c_m(-l_{mi})$ and $b_m(-l_{mi})$ on any Gelfand-Tsetlin module listed below.

\begin{corollary}\label{relations} Let  $L=(l_{ij})$  and $\sm=\sm_L$. Let $V$ be a Gelfand-Tsetlin module with $V_{\sm}\neq 0$ and $v\in V_{\sm}$ a nonzero element such that $\sm v=0$. Then the following hold.
\begin{itemize}
\item[(i)]  $[c_m(x), b_l(y)] V_{\sm}= 0$ for all $x$, $y$ if $l\neq m$;
\item[(ii)] $[s_m(x), s_l(y)] V_{\sm}= 0$ or all  $x$, $y$, $s\in \{b, c\}$ if $|l -m| \neq 1$;
\item[(iii)] $[c_m(-l_{mi}), b_m(-l_{mj})]v=0$ if $l_{mi}\neq l_{mj}$;
\item[(iv)] $b_m(-l_{mi}+1)c_m(-l_{mi}) v= - a_{m+1}(-l_{mi}+1)a_{m-1}(-l_{mi})v$;
\item[(v)] $
c_m(-l_{mi}-1)b_m(-l_{mi})v= - a_{m+1}(-l_{mi})a_{m-1}(-l_{mi}-1)v$;
\item[(vi)]  $(x - y)  [a_m(x), c_m(y)] w= c_m(x)a_m(y)w - c_m(y)a_m(x)w$ for any $w\in V_{\sm}$;
\item[(vii)] $(x - y)  [a_m(x), b_m(y)] w=- b_m(x)a_m(y) w+ b_m(y)a_m(x)w$ for any $w\in V_{\sm}$. 
\end{itemize}
\end{corollary}

Recall that a Gelfand-Tsetlin module $V$ is \emph{tame}, if the action of $\Gamma$ on $V$ is diagonalizable.  In \cite{FGR2} it is proven that an irreducible  $1$-singular Gelfand-Tsetlin module that has no critical maximal ideals  in its Gelfand-Tsetlin support is tame. We now generalize this to any singular Gelfand-Tsetlin module.

\begin{proposition}\label{prop--tame}
If $V$ is an irreducible Gelfand-Tsetlin module such that $V$ has no critical  maximal ideals in its Gelfand-Tsetlin support then $V$ is a tame module.
\end{proposition}

\begin{proof}
 Fix $\sm$ in the Gelfand-Tsetlin support of $V$ and choose any nonzero $v\in V_{\sm}$ 
 such that $\sm v=0$. Let $\sm=\sm_L$ and $L=(l_{ij})$.
Applying Corollary \ref{relations}(vi)-(vii) we obtain that $\sm_{L-\delta_{mi}}c_m(-l_{mi})v=0$ and $\sm_{L+\delta_{mi}}b_m(-l_{mi})v=0$, that is, $c_m(-l_{mi})v$ and $b_m(-l_{mi})v$ are eigenvectors of $\Gamma$.  Since the action of $U(\gl_n)$ on $V$ is determined by the actions of $a_m(-l_{mi})$, $b_m(-l_{mi})$, $c_m(-l_{mi})$ on $V$ (for all possible values of $m$ and $i$ and all maximal ideals in the Gelfand-Tsetlin support of $V$), 
 the subspace of $V$ consisting of eigenvectors of $\Gamma$ is $U(\gl_n)$-invariant. This completes the proof.\end{proof}

\subsection{Drinfeld quiver and Drinfeld category} 
Let $\sm$ be a non-critical ideal in $\Gamma$ and let $L$ be such that $\sm = \sm_L$. Let also ${\mathcal H}_{i_1,i_2}^m = \{ (z_{ij}) \in T_{n-1} (\mathbb C) \; | \; z_{mi_1} = z_{mi_2} \}$ be the critical hyperplane in $T_{n-1} ({\mathbb C})$ corresponding to the pair $(z_{mi_1}, z_{mi_2})$. Let $\mathcal{NC}_{\mathbb Z} (L)$ be the connected component of $L_{\mathbb Z} \setminus \bigcup_{i_1,i_2,m} {\mathcal H}_{i_1,i_2}^m $ containing $L$ (we call a subset $A$ of $L_{\mathbb Z}$ connected if $A = A' \cap L_{\mathbb Z}$ for some some connected subset $A'$ of $T_{n-1} (\mathbb C)$).
Then let $\mathcal{NC}_{\sm}$  be the set of ideals $\sm_{L'}$ in $\Gamma$ such that $L' \in  \mathcal{NC}_{\mathbb Z} (L)$ (we identify maximal ideals that correspond to the same orbit of the group $S_1 \times \ldots \times S_n$). One easily see that $\mathcal{NC}_{\sm}$ is independent on the choice of $L$ in a connected component of $L_{\mathbb Z} \setminus \bigcup_{i_1,i_2,m} {\mathcal H}_{i_1,i_2}^m$.

We define the (non-critical) Drinfeld quiver with relations $\mathcal D_{\sm}$ as follows. The set of vertices of $\mathcal D_{\sm}$ is  $\mathcal{NC}_{\sm}$. Recall the definition of $\sm_{L' \pm \delta_{mi}}$ prior to Proposition \ref{prop-der-crit}. If  $L'=(l'_{ij})$, then   $\omega_{L'}\in \mathcal{NC}_{\sm}$ is the start point of edges  $a_k(-l'_{mi})$, $b_m(-l'_{mi})$, $c_m(-l'_{mi})$,  $i=1, \ldots, m$; $m=1, \ldots, n-1$.  The end points of these edges are determined by the properties in Proposition \ref{prop-der-crit}. More precisely, the end points of the edge $b_m(-l'_{mi})$ and $c_m(-l'_{mi})$ are $\sm_{L' + \delta_{mi}}$ and $\sm_{L' - \delta_{mi}}$, respectively. Furthermore, the end point of  $a_k(-l'_{mi})$ coincides with its start point. We now describe the relations of ${\mathcal D}_{\sm}$. 
We first include those relations obtained from  the relations 
 of Corollary \ref{relations} after evaluating $u$ at $-l'_{mi}$. In addition, we add the relations $a_k(-l_{mi}') e_{\sm_{L'}} = \alpha_k(-l_{mi}', \sm) e_{\sm_{L'}} $, where $e_{\sm_{L'}} $ is the idempotent of $\mathcal D_{\sm}$ corresponding to the vertex $\sm_{L'}$ (recall the definition of  $\alpha_k(u, \sm)$ in \S \ref{subsec-interpol}). In particular, all $a_k(-l_{mi}')$ pairwise commute. Due to the Lagrange interpolation formula it is sufficient to consider only the morphisms $a_m(-l'_{mi})$ at the vertex 
 $\sm_{L'}$.

 \begin{example}
Let $n=3$ and $\tilde{L}$ be a $1$-singular critical tableau, i.e. $\tilde{l}_{21} = \tilde{l}_{22}$. 
Consider the critical plane ${\mathcal H} = {\mathcal H}_{1,2}^2$ in $T_2(\mathbb C) = {\mathbb C}^3$. Let $L \in \tilde{L}_{\mathbb Z}$ be a non-critical tableau and let $\sm = \sm_{L}$. The hyperplane ${\mathcal H}$ separates the set $L_{\mathbb Z} =  \tilde{L}_{\mathbb Z}$ into two  (discrete) half-spaces and $\mathcal{NC}_{\mathbb Z}(L)$ is the one containing $L$. Then the set of vertices of  ${\mathcal D}_{\sm}$  is 
$\{\sm_{L'}  \; | \; L'  \in \mathcal{NC}_{\mathbb Z}(L) \}$. Below we present a partial picture of ${\mathcal D}_{\sm}$ involving the edges $b_2(-l_{2,i}')$ and $c_2(-l_{2,i}')$, $i=1,2$. Note that the complete picture is 3-dimensional half lattice (avoiding the critical plane $\mathcal H$) and involves also the edges  $a_k(-l_{mi}')$, $b_1(-l_{1,1}')$,  $c_1(-l_{1,1}')$.
\end{example}
 $$
\xymatrixcolsep{4pc} \xymatrixrowsep{3pc} \xymatrix{ \sm_{L'- \delta_{2,1} + \delta_{2,2}}  \ar@<0.5ex>[d]^{c_2 (-l_{2,2}'-1)}    \ar@<0.5ex>[r]^{b_2 (-l_{2,1}'+1)} 
& \sm_{L' + \delta_{2,2}}   \ar@<0.5ex>[d]^{c_2 (-l_{2,2}'-1)}  \ar@<0.5ex>[l]^{c_2 (-l_{2,1}')}  \ar@<0.5ex>[r]^{b_2 (-l_{2,1}')} 
& \sm_{L'+ \delta_{2,1} + \delta_{2,2} }   \ar@<0.5ex>[d]^{c_2 (-l_{2,2}'-1)} \ar@<0.5ex>[l]^{c_2 (-l_{2,1}'-1)}  \\  \sm_{L'-\delta_{2,1}}  \ar@<0.5ex>[r]^{b_2 (-l_{2,1}'+1)}  \ar@<0.5ex>[u]^{b_2 (-l_{2,2}')}  & \sm_{L'}   \ar@<0.5ex>[l]^{c_2 (-l_{2,1}')}  \ar@<0.5ex>[r]^{b_2 (-l_{2,1}')}  \ar@<0.5ex>[u]^{b_2 (-l_{2,2}')} &  \sm_{L'+\delta_{2,1} }    \ar@<0.5ex>[l]^{c_2 (-l_{2,1}'-1)}  \ar@<0.5ex>[u]^{b_2 (-l_{2,2}')} 
}
$$
 
 \medskip
 
  Recall that the purpose of the paper is to classify all simple $1$-singular Gelfand-Tsetlin modules $V$. In the case when $V$ has a 2-dimensional Gelfand-Tsetlin subspace $V_{\sm}$ for some $\sm$, then $V$ is a subquotient of the module $V(L)$ for some $L$, \cite{FGR2}. Thus, to complete our classification of simple $1$-singular Gelfand-Tsetlin modules we  need only  to consider  tame modules only.
 
We fix again a non-critical ideal $\sm$ in $\Ga$. Denote by $\mathcal D_{\sm}$-mod the category of representations of the quiver with relations $\mathcal D_{\sm}$. We call $\mathcal D_{\sm}$-mod the \emph{Drinfeld category} associated to $\sm$. If $V\in  \mathcal D_{\sm}$-mod  the vector space associated to the vertex  $\omega \in \mathcal {NC}_{\sm}$ will be   denoted by $V_{\omega}$. 
 If $V\in  \mathcal D_{\sm}$-mod then the \emph{Gelfand-Tsetlin support} of $V$ is the set of all vertices $\omega \in \mathcal {NC}_{\sm}$ for which $V_{\omega} \neq 0$.

Let $\mathcal{NCGT}_{\sm}$ be the full subcategory of the category of all Gelfand-Tsetlin modules consisting of  Gelfand-Tsetlin modules
whose Gelfand-Tsetlin support is a subset of $\mathcal {NC}_{\sm}$. In particular, the modules in  $\mathcal{NCGT}_{\sm}$ 
 have no critical tableau in their support, i.e.  are non-critical.
Recall from Proposition \ref{prop--tame} that any irreducible module in $\mathcal{NCGT}_{\sm}$ is tame. Denote by $\mathcal{NCT}_{\sm}$ 
 the subcategory of $\mathcal{NCGT}_{\sm}$ consisting of tame modules.

By the definition of $\mathcal D_{\sm}$, any module in $\mathcal {NCT}_{\sm}$ can be naturally viewed as a module over $\mathcal D_{\sm}$. 
As a result we have a  functor $$F_{NC}: \mathcal {NCT}_{\sm}\rightarrow \mathcal D_{\sm}\rm{-mod}.$$   
  In fact, $ \mathcal {NCT}_{\sm}$ is equivalent to a quotient category of $\mathcal D_{\sm}$-mod.

The following proposition is a direct consequence of the definitions of $\mathcal {NCT}_{\sm}$ and $\mathcal{D}_{\sm}$ and Corollary \ref{relations}.

\begin{proposition}\label{prop-categories}
The functor $F_{NC}$ is faithful and maps irreducible modules to irreducible. Moreover, the images of non-isomorphic modules are non-isomorphic.
\end{proposition}



\begin{remark}
One can generalize the notion of Drinfeld quiver ${\mathcal D}_{\sm}$ to arbitrary ideals $\sm \in \Ga$ and define a general functor $F$ from a more general category of Gelfand-Tsetlin modules (not necessarily tame) to $\mathcal D_{\sm}$-mod. For this we need to introduce new relations on  ${\mathcal D}_{\sm}$ that involve higher-order derivatives of $a_m(u)$, $b_m(u)$, $c_m(u)$. We will treat this case in a subsequent paper. 
\end{remark}

\section{Irreducible   $\mathcal D_{\sm}$-modules}

 For  $\omega \in \mathcal{NC}_{\sm}$, denote by 
$\mathcal D_{\sm}(\omega) = \mathcal D_{\sm}(\omega, \omega)$ the algebra of all paths of $\mathcal D_{\sm}$ originating and ending at $\omega$.  If $V\in  \mathcal D_{\sm}$-mod ,  the subspace   $V_{\omega}$ of $V$ has a natural $\mathcal D_{\sm}(  \omega)$-module structure.
We have the following standard result which can be proven using induction and restriction functors.

\begin{lemma}\label{lem-cyclic}
\begin{itemize}
\item[(i)] For any $\omega \in \mathcal{NC}_{\sm}$ and any irreducible module $V$ of $\mathcal D_{\sm}$,  the $\mathcal D_{\sm}(  \omega)$-module $V_{\omega}$ is either zero or irreducible.
\item[(ii)] For any $\omega \in \mathcal{NC}_{\sm}$ and  any irreducible module $N$ over $\mathcal D_{\sm}(  \omega)$ there exists a unique irreducible module of $\mathcal D_{\sm}$ whose restriction on $\mathcal D_{\sm}(  \omega)$ is isomorphic to $N$.
\end{itemize}
\end{lemma}

\subsection{Irreducible tame  $\mathcal D_{\sm}$-modules}

Let $\sm$ be a maximal ideal of $\Ga$. In this subsection we will analyze irreducible modules over the Drinfeld category   $\mathcal D_{\sm}$. 
The theorem below will  be used in the  proof of Main Theorem.

\begin{theorem}\label{the-non-critical} Let $\sm$ be a   non-critical maximal ideal in $\Gamma$ and
 $V$ be an irreducible Gelfand-Tsetlin module  such that $V$ has no critical 
  maximal ideals in its Gelfand-Tsetlin support.
Then for any $\omega$ in the Gelfand-Tsetlin support of $V$,  $V_{\omega}$ is a $1$-dimensional module over $\mathcal D_{\sm}( \omega)$ whose module structure is completely determined by $\omega$. 
\end{theorem}

\begin{proof} By Proposition \ref{prop--tame} the module $V$ is tame.  Without loss of generality we may assume that $\omega = \sm$. 

Let $\sm=\sm_L$ and $L=(l_{ij})$. Take any nonzero $v \in V_{\sm}$. Then $v$ is an eigenvector of all $a_k(-l_{mi})$. We will show that  the dimension of the space $ \mathcal D_{\sm}(\sm)v$ is $1$. This will imply the statement of the theorem thanks to  Lemma \ref{lem-cyclic}.

Consider an arbitrary nonzero element $u\in \mathcal D_{\sm}(\sm)$. Then $u$ is a linear combination of monomials formed by $b_i(t)$ and $c_i(t)$, $i=1, \ldots, n$ for different values of $t$. Without loss of generality we may assume that $u$ equals a single monomial.  The first element (reading from the right) of this monomial is one of the  operators $c_m(-l_{mk}), b_m(-l_{mk})$ for some  
 $m=1, \ldots, n$ and some  $k=1, \ldots, m$.  Suppose that this element is $c_m(-l_{mk})$.  Without loss of generality we may assume that $k=1$. Then the monomial contains the  elements
  $$c_m[l_{m1}, t]:=c_m(-l_{m1}+t)\ldots c_m(-l_{m1}+1)c_m(-l_{m1}),$$
for some $t \geq 0$.    Note that since 
$V$ is non-critical, for $j \neq 1$,  $c_m[l_{m1}, t]$ and $c_m[l_{mj}, s]$ commute by Corollary \ref{relations}.  Hence, we can rewrite the rightmost part of the monomial as

$$\pi_c^1:=\overrightarrow{\Pi}_{i=1, \ldots, n-1}c_i[l_{i1}, t_{i1}]\ldots c_i[l_{ii}, t_{ii}],$$
for some $t_{ij}\in \{-1, 0, \ldots\}$. By definition, $c_i[l_{ij}, -1]=1$. Since our monomial is an element of $\mathcal D_{\sm}(\sm)$ it must contain 
 respective elements $b_i(x_i)$. Suppose  $b_m(-r_{mk})$ is the first such element that appears in the monomial. Then $\pi_c^1 V_{\sm}\subset V_{\sn}$, $\sn=\sn_R$ and  
 $R=(r_{ij})$.  Let   $r_{mk}=l_{mk}-t$ for some nonnegative integer $t$.  
 
 Assume first that $t=0$.  Then   $t_{mk}=-1$ and hence $b_m(-r_{mk})$ commutes with $\pi_c^1$ and 
$b_m(-r_{mk})\pi_c^1\notin \mathcal D_{\sm}(\sm)$. 

Set also  
$$
b_m[x, r]:=b_m(-x+r)\ldots b_m(-x+1)b_m(-x)$$
 with the same convention for $r=-1$ as above. Then we may assume that our monomial contains 
$$\pi_b^1=\overrightarrow{\Pi}_{i=1, \ldots, n-1}b_i[l_{i1}, r_{i1}]\ldots b_i[l_{ii}, r_{ii}],$$
for some $r_{ij}\in \{-1, 0, \ldots\}$.  Finally, we can assume that our monomial contains

$$\Pi_j \pi_c^j\Pi_s \pi_b^s,$$
such that every $\pi_c^j$ commutes with every $\pi_b^s$.  We see again that $\Pi_j \pi_c^j\Pi_s \pi_b^s \notin  \mathcal D_{\sm}(\sm)$. 

Assume now that $t>0$. Then $\pi_c^1$ must contain $c_m(-r_{mk}-1)$. The action of $b_m(-r_{mk})$
commutes with the action of $c_m(x)$ if $x\neq -r_{mk}$  and with the action of any $c_s(y)$ if $s\neq m$ by Corollary  \ref{relations}. Thus,  we can move $b_m(-r_{mk})$ next to $c_m(-r_{mk}-1)$ in the monomial and then apply  Corollary  \ref{relations} again. As a result,     
$b_m(-r_{mk})c_m(-r_{mk}-1)$ can be replaced by a scalar and can be removed from the product. Since the original monomial is  in $\mathcal D_{\sm}(\sm)$, using the procedure described above,  we will eventually remove all 
$b$'s and all $c$'s. This shows that $\mathcal D_{\sm}(\sm) v$ is $1$-dimensional. Applying Lemma \ref{lem-cyclic} and the fact that $V$ is  irreducible, we conclude that 
$V_{\sm}$ is $1$-dimensional and its structure as a $\mathcal D_{\sm}(\sm)$-module is completely determined by $\sm$. This completes the proof of the theorem.  \end{proof}



Combining Lemma \ref{lem-cyclic}  and Theorem \ref{the-non-critical}  we  obtain the following.

\begin{corollary}\label{cor-uniqueness}
For any non-critical maximal ideal $\sm$ of $\Ga$ there exists at most one (up to an isomorphism)
 irreducible module in the category $\mathcal {D}_{\sm}{\rm -mod}$  such that $V$ has no critical 
 maximal ideals in its Gelfand-Tsetlin support.

\end{corollary}

\begin{remark}
Let $\sm$  be a generic maximal ideal and $M$  be an irreducible Gelfand-Tsetlin module over $\gl_n$ that has $\sm$ in its Gelfand-Tsetlin  support. Then  $M$ is non-critical. Let $F_{NC}(M) = V$. Since any maximal ideal $\sm'$ in the Gelfand-Tsetlin 
support of $M$ is generic we have $\dim M_{\sm'}=1$, implying $\dim V_{\sm'}=1$ \cite{FGR1}.  Moreover, $M$ is a unique Gelfand-Tsetlin module that has $\sm$ in its Gelfand-Tsetlin  support.  Corollary \ref{cor-uniqueness} gives an alternative proof of this result.
\end{remark}


\section{Classification of  irreducible $1$-singular  Gelfand-Tsetlin modules}

If $V$ contains a critical maximal ideal in its Gelfand-Tsetlin support, then $V$ need not to be a tame module.  Nevertheless we have the following.

\begin{theorem}\label{the-critical}[\cite{FGR2}, Theorem 7.2(ii)]\label{thm-critical} Let  $M$ be an irreducible Gelfand-Tsetlin module   such that $V$ contains a  critical $1$-singular
  maximal ideal $\sm$  of $\Ga$ in its Gelfand-Tsetlin support. Then
  $M_{\sm}$ is  $1$-dimensional. Moreover, $M$ is a unique such module with  $M_{\sm}\neq 0$. 
In particular, there exists at most one (up to an isomorphism) such irreducible tame module. 
\end{theorem}

Now we are ready to prove the conjecture stated in the introduction and, in particular, complete the classification of irreducible $1$-singular modules.

\begin{theorem}\label{the-isolated}
Let $M$ be an irreducible $1$-singular Gelfand-Tsetlin $\gl_n$-module, and $\sm=\sm_L$ be  a maximal ideal of $\Ga$ in the Gelfand-Tsetlin  support of $M$. 
Then $M$ is isomorphic to a subquotient of $V(L)$.
\end{theorem}

\begin{proof}
 Consider a $1$-singular maximal ideal $\sm=\sm_L$. If $\sm$
 is critical then such  module $M$  is unique by Theorem \ref{the-critical} and, hence, it is isomorphic to a subquotient of $V(L)$. Suppose now that $\sm$ is not critical. Then 
 $V(L)$ has at least one irreducible subquotient $S$  with $\sm$ in its Gelfand-Tsetlin support. Recall that all Gelfand-Tsetlin multiplicities of $U_n/U_n\sm$ are at most $2$ by \cite{FGR2}, Corollary 7.1 (cf. also \cite{FO}, Theorem 4.12). Also note that any irreducible Gelfand-Tsetlin $\gl_n$-module $N$ that has $\sm$ in its Gelfand-Tsetlin support is a subquotient of $U_n/U_n\sm$. Then we immediately have that 
 $M\simeq S$ if  $\dim S_{\sm}=2$. Assume now that $\dim S_{\sm}=1$. Then there exists another irreducible subquotient $S'$ of $V(L)$ with $\sm$ in the Gelfand-Tsetlin support. If $S'$ is not isomorphic to $S$ then $M$ is isomorphic to one of them since both these modules are subquotients of $U_n/U_n\sm$. It remains to consider the case when 
  $S'\simeq S$. Suppose $M$ is not isomorphic to $S$. Then 
  $M$ can not have any critical maximal ideals  in its Gelfand-Tsetlin support  by Theorem \ref{thm-critical}.  Indeed, otherwise $M$ is a subquotient of $V(L)$ and thus isomorphic to $S$. 
  Applying Proposition \ref{prop--tame}  we conclude that $M$ is a tame module. But an irreducible tame  Gelfand-Tsetlin module with no critical maximal ideals  that has $\sm$ in the Gelfand-Tsetlin support is unique (up to an isomorphism) by Theorem \ref{the-non-critical}. Therefore, $M\simeq S$ which completes the proof. 
 \end{proof}


\begin{thebibliography}{20}
\bibitem{CE1} M. Colarusso, S.Evens, On algebraic integrability of Gelfand-Zeitlin fields, Transform.
Groups 15 (2010), no. 1, 46--71.

\bibitem{CE2} M. Colarusso, S.Evens, The Gelfand-Zeitlin integrable system and K-orbits on the
flag variety, in: Symmetry: Representation Theory and its Applications:, Progress 
Math. Birkauser, Boston, 2014, 36 pages.

\bibitem{D} V.G. Drinfeld, A new realization of Yangians and quantized affine algebras, Soviet
Math. Dokl. 36 (1988), 212--216.

\bibitem{DFO2} Y. Drozd, S. Ovsienko, V. Futorny, Irreducible weighted $\mathfrak{sl}(3)$-modules, {Funktsional. Anal. i Prilozhen.},  23 (1989), 57--58.


\bibitem{DFO3} Y. Drozd, S. Ovsienko, V. Futorny,  Harish-Chandra subalgebras and Gelfand-Zetlin modules, {Math. Phys. Sci.} {424} (1994), 72--89.

\bibitem{FO} V. Futorny, S. Ovsienko, Fibers of characters in Gelfand-Tsetlin categories,  {Trans. Amer. Math. Soc.} {366} (2014), 4173--4208.

\bibitem{FGR1} V. Futorny, D. Grantcharov, L.E. Ramirez, Irreducible Generic Gelfand-Tsetlin Modules of gl(n), SIGMA 11, (2015) 018, 13 pp.



\bibitem{FGR2} V. Futorny, D. Grantcharov, L. E. Ramirez, Singular Gelfand-Tsetlin modules for $\mathfrak{gl}(n)$, Adv. Math., 290, (2016), 453--482.


\bibitem{FGR4} V. Futorny, D. Grantcharov, L. E. Ramirez, New singular Gelfand-Tsetlin $gl(n)$-modules of index $2$,  to appear in Comm. Math. Phys.

\bibitem{FGR3} V. Futorny, D. Grantcharov, L. E. Ramirez, Classification of irreducible Gelfand-Tsetlin modules for $\mathfrak{sl}(3)$. In progress.





\bibitem{GG} I. Gelfand, M. Graev, Finite-dimensional irreducible representations of the unitary and complete linear group and special functions associated with them, {Izv. Akad. Nauk SSSR Ser. Mat.} {29} (1965), 1329--1356. (in Russian).




\bibitem{GT} I. Gelfand, M. Tsetlin, Finite-dimensional representations of the group of unimodular matrices, {Dokl. Akad. Nauk SSSR (N.s.)}, {71} (1950), 825--828.

\bibitem{GR} C. Gomez, L.E. Ramirez, Families of irreducible singular Gelfand-Tsetlin modules of $\mathfrak{gl}(n)$, arXiv:1612.00636 

\bibitem{Gr1} M. Graev, Infinite-dimensional representations of the Lie algebra $\gl(n, \mathbb{C})$
related to complex analogs of the Gelfand-Tsetlin patterns and
general hypergeometric functions on the Lie group $GL(n, \mathbb{C})$, {Acta Appl. Math.}, {81} (2004), 93--120.

\bibitem{Gr2} M. Graev, A continuous analogue of Gelfand-Tsetlin schemes and a realization of the
principal series of irreducible unitary representations of the group $GL(n,\mathbb{C})$ in the space
of functions on the manifold of these schemes, {Dokl. Akad. Nauk.} {412 no.2} (2007), 154--158.

\bibitem{GS}V. Guillemin and S. Sternberg, The Gelfand-Tsetlin system and quantization of the complex 
flag manifolds, J. Funct. Anal. 52 (1983), no. 1, 106-128.


\bibitem{KW-1}  B. Kostant,  N. Wallach, Gelfand-Zeitlin theory from the perspective of
classical mechanics I, in: Studies in Lie Theory Dedicated
to A. Joseph on his Sixtieth Birthday, {Progress in
Mathematics},  {243} (2006),   319--364.


\bibitem{KW-2}  B. Kostant,  N. Wallach, Gelfand-Zeitlin theory from the perspective of
classical mechanics II, in: The Unity of Mathematics In
Honor of the Ninetieth Birthday of I. M. Gelfand, {Progress
in Mathematics},  {244} (2006), 387--420.

\bibitem{Maz1} V. Mazorchuk, Tableaux realization of generalized Verma modules, {Can. J. Math.}  {50} (1998), 816--828.


\bibitem{Maz2} V. Mazorchuk, On categories of Gelfand-Zetlin modules, in: {Noncommutative Structures in Mathematics and Physics}, Kluwer Acad. Publ, Dordrecht (2001),  299--307.

\bibitem{m:gtsb} A. Molev,  Gelfand-Tsetlin bases for classical Lie algebras, in: {Handbook of Algebra, Vol. 4, (M. Hazewinkel, Ed.), Elsevier,} (2006), 109--170.

\bibitem{MNO} A. Molev, M. Nazarov and G. Olshanski, Yangians and classical Lie algebras,
Russian Math. Surveys 51 (1996), 205--282.

\bibitem{NT} M.Nazarov, V.Tarasov, Yangians and Gelfand-Zetlin bases, Publ. RIMS, Kyoto Univ.
30 (1994), 459--478.

\bibitem{Ovs} S. Ovsienko, Finiteness
statements for Gelfand-Zetlin modules, in: {Third International Algebraic Conference in the Ukraine (Ukrainian)},  Natsional. Akad. Nauk Ukrainy, Inst. Mat., Kiev, 2002,  323--338.


\bibitem{Zh} D. Zhelobenko, Compact Lie groups and their representations, {Transl. Math. Monographs, vol. 40, AMS}, 1974.


\end{thebibliography}
\end{document}